\documentclass[11pt]{amsart}
\usepackage{amsmath, amssymb, amsfonts, verbatim, graphicx, amsthm, tikz-cd, mathrsfs, color, comment, fancyhdr, amsrefs}
\usepackage[top=1in, bottom=1in, left=1in, right=1in]{geometry}
\usepackage{enumerate}

\usepackage{mathtools}

\usepackage{url}

\usepackage{enumitem}

\usepackage[foot]{amsaddr}

\allowdisplaybreaks[1]

\usepackage[sc]{mathpazo}
\usepackage[T1]{fontenc}

\newcommand{\vertiii}[1]{{\left\vert\kern-0.25ex\left\vert\kern-0.25ex\left\vert #1
		\right\vert\kern-0.25ex\right\vert\kern-0.25ex\right\vert}}

\theoremstyle{plain}
\newtheorem*{Thm}{Theorem}

\theoremstyle{definition}
\newtheorem*{Def}{Definition}

\theoremstyle{remark}
\newtheorem*{Rmk}{Remark}

\renewcommand{\epsilon}{\varepsilon}

\usepackage{setspace}

\usepackage{hyperref}

\title{A nonstandard-analytic proof of a theorem regarding noncommutative ergodic optimizations}

\author{Aidan Young$^1$}
\address{University of North Carolina at Chapel Hill}

\email{$^1$\url{aidanjy@live.unc.edu}}
\begin{document}
\maketitle

\begin{abstract}
In \cite{Number2}, we extended the notion of ergodic optimization to the setting of C*-dynamical systems of countable discrete groups. Among the key results of that paper was that given an action $G \stackrel{\Xi}{\curvearrowright} \mathfrak{M}$ of a countable discrete amenable group $G$ on a W*-probability space $(\mathfrak{M}, \rho)$ by $\rho$-preserving $*$-automorphisms of $\mathfrak{M}$, a positive element $x \in \mathfrak{M}$, and a right Følner sequence $\mathcal{F} = (F_k)_{k \in \mathbb{N} }$ for $G$, the sequence
$$\left( \left\| \frac{1}{|F_k|} \sum_{g \in F_k} \Xi_g x \right\| \right)_{ k \in \mathbb{N} }$$
converges to a value $\Gamma(x)$ which can be described in the language of ergodic optimization. We provide here an alternate, more direct proof of that theorem using the tools of nonstandard analysis.
\end{abstract}

In \cite{Number2}, we extended the basic concepts of ergodic optimization from topological dynamics (see \cite{Jenkinson}, \cite{JenkinsonEvery}) to the study of C*-dynamical systems. In that same paper, we also extended the gauge map introduced in \cite{Assani-Young} to the context of C*-dynamical systems, and demonstrated that in this more general setting, the gauge corresponds to an ergodic maximum \cite[Theorem 5.3]{Number2}. Though our proof of that result in that paper is inspired by the techniques of nonstandard analysis, the proof presented there does not explicitly use the language of nonstandard analysis, and as a result, the argument is more roundabout than it needs to be. We provide here a proof of that same result, but making explicit use of the tools of nonstandard analysis.

\begin{Def}
Let $G$ be a countable discrete group. We call $G$ \emph{amenable} if there exists a sequence $(F_k)_{k \in \mathbb{N} }$ of nonempty finite subsets of $G$ such that
$$\lim_{k \to \infty} \frac{\left| F_k \Delta F_k g \right|}{|F_k|} = 0$$
for all $g \in G$. We call this sequence $(F_k)_{ k \in \mathbb{N} }$ a \emph{Følner sequence} for $G$.
\end{Def}

\begin{Rmk}
What we call here a Følner sequence for $G$ would be referred to in other literature as a \emph{right Følner sequence} for $G$, as opposed to a \emph{left Følner sequence} for $G$, which would satisfy the identity
\begin{align*}
\lim_{k \to \infty} \frac{\left| F_k \Delta g F_k \right|}{|F_k|}	& = 0	& (\forall g \in G) .
\end{align*}
Since we're only concerned here with right Følner sequences, we simply refer to them as Følner sequences.
\end{Rmk}

We define a \emph{C*-dynamical system} to be a triple $(\mathfrak{A}, G, \Theta)$ consisting of a unital C*-algebra $\mathfrak{A}$, a countable discrete amenable group $G$, and an action $\Theta$ of $G$ on $\mathfrak{A}$ by *-automorphisms. Denote by $\mathcal{S}^G$ the weak*-compact convex set of states $\phi$ on $\mathfrak{A}$ satisfying the $\Theta$-invariance property
\begin{align*}
\phi	& = \phi \circ \Theta_g	& (\forall g \in G) .
\end{align*}
Since $G$ is amenable, the set $\mathcal{S}^G$ is nonempty. Given a compact, convex subset $K$ of $\mathcal{S}^G$ and positive $a \in \mathfrak{A}$, we define the value
$$m \left( a \vert K \right) : = \sup_{\phi \in K} \phi(a) .$$

We define a \emph{W*-dynamical system} to be a quadruple $(\mathfrak{M}, \rho, G, \Xi)$, where $\mathfrak{M}$ is a W*-dynamical system equipped with a faithful normal state $\rho$, and $\Xi$ is an action of the countable discrete amenable group $G$ on $\mathfrak{M}$ by *-automorphisms satisfying the identity
\begin{align*}
\rho	& = \rho \circ \Xi_g	& (\forall g \in G) .
\end{align*}
Finally, we define a \emph{C*-model} of a W*-dynamical system $(\mathfrak{M}, \rho, G, \Xi)$ to be a quadruple $(\mathfrak{A}, G, \Theta ; \iota)$, where $(\mathfrak{A}, G, \Theta)$ is a C*-dynamical system and $\iota : \mathfrak{A} \to \mathfrak{M}$ is a *-morphism satisfying the equivariance identity
\begin{align*}
\Xi_g \circ \iota	& = \iota \circ \Theta_g	& (\forall g \in G) ,
\end{align*}
and such that the image $\iota(\mathfrak{A})$ is dense in $\mathfrak{M}$ with respect to the weak operator topology. We call the C*-model \emph{faithful} if $\iota$ is an injective map.

\begin{Thm}\label{Gamma converges}
Let $(\mathfrak{M}, \rho, G, \Xi)$ be an ergodic W*-dynamical system, and let $(\mathfrak{A}, G, \Theta; \iota)$ be a C*-model of $(\mathfrak{M}, \rho, \mathbb{Z}, \Xi)$. Let $\mathcal{F} = (F_k)_{k \in \mathbb{N} }$ be a Følner sequence for $G$. Then if $a \in \mathfrak{A}$ is a positive element, then the sequence $\left( \left\| \frac{1}{|F_k|} \sum_{g \in F_k} \Xi_g \iota(a) \right\| \right)_{k \in \mathbb{N} }$ converges, and its limit is
	$$m \left( a \vert \operatorname{Ann}(\ker \iota) \right) .$$
\end{Thm}

\begin{proof}
For each $k \in \mathbb{N}$, let $\sigma_k$ be a state on $\mathfrak{M}$ such that
$$\sigma_k \left( \frac{1}{|F_k|} \sum_{g \in F_k} \Xi_g \iota(a) \right) = \left\| \frac{1}{|F_k|} \sum_{g \in F_k} \Xi_g \iota(a) \right\| . $$
This state exists because $\frac{1}{|F_k|} \sum_{g \in F_k} \Xi_g \iota(a)$ is positive for all $k \in \mathbb{N}$, and for any $x \in \mathfrak{M}$ exists a state $\phi$ such that $\left| \phi(x) \right| = \| x \|$ (a standard result, see e.g. \cite[II.6.3.3]{Blackadar}). Fix an unlimited hypernatural $K \in \prescript{*}{}{\mathbb{N}} \setminus \mathbb{N}$, and define a state $\omega^{(K)}$ on $\mathfrak{M}$ as follows.
Let $\mathbb{L} : = \bigcup_{n \in \mathbb{N} } \left\{ z \in \prescript{*}{}{\mathbb{C}} : |z| \leq n \right\}$ be the complex algebra of all limited hypercomplex numbers, and let $\operatorname{sh} : \mathbb{L} \to \mathbb{C}$ be the shadow functional, i.e. $\operatorname{sh}(z)$ is the unique $z' \in \mathbb{C}$ such that $\left| z - z' \right| < \frac{1}{n}$ for all $n \in \mathbb{N}$. Given $x \in \mathfrak{M}$, let $f_x : \mathbb{N} \to \left\{ z \in \mathbb{C} : |z| \leq \| x \| \right\}$ be the (standard) sequence
$$f_x(k) = \sigma_k \left( \frac{1}{|F_k|} \sum_{g \in F_k} \Xi_g \iota(a) \right) ,$$
and let $\prescript{*}{}{f}_x : \prescript{*}{}{\mathbb{N}} \to \left\{ z \in \prescript{*}{}{\mathbb{C}} : |z| \leq \| x \| \right\}$ be the hypersequence extending $f_x$. Finally, set
$$\omega^{(K)}(x) = \operatorname{sh} \left( \prescript{*}{}{f}_x(K) \right) .$$
Since $\operatorname{sh} : \mathbb{L} \to \mathbb{C}$ is a complex-linear functional, we know that $\omega^{(K)}$ is a state.

We can also check that $\omega^{(K)}$ is $\Xi$-invariant. Fix $g_0 \in G$. Then for $k \in \mathbb{N}$, we have
\begin{align*}
	\left| f_x(k) - f_{\Xi_{g_0} x}(k) \right|	& = \frac{1}{|F_k|} \left| \left( \sum_{g \in F_k} \sigma_k(\Xi_g \Xi_{g_0} x) \right) - \left( \sum_{g \in F_k} \sigma_k (\Xi_g x) \right) \right| \\
	& = \frac{1}{|F_k|} \left| \left( \sum_{g \in F_k} \sigma_k(\Xi_{g g_0} x) \right) - \left( \sum_{g \in F_k} \sigma_k (\Xi_g x) \right) \right| \\
	& = \frac{1}{|F_k|} \left| \left( \sum_{g \in F_k g_0} \sigma_k(\Xi_g x) \right) - \left( \sum_{g \in F_k} \sigma_k (\Xi_g x) \right) \right| \\
	& = \frac{1}{|F_k|} \left| \left( \sum_{g \in F_k g_0 \setminus F_k} \sigma_k(\Xi_g x) \right) - \left( \sum_{g \in F_k \setminus F_k g_0} \sigma_k (\Xi_g x) \right) \right| \\
	& \leq \frac{1}{|F_k|} \left( \left(\sum_{g \in F_k g_0 \setminus F_k} \|x\| \right) + \left(\sum_{g \in F_k \setminus F_k g_0} \|x\| \right) \right) \\
	& = \frac{|F_k \Delta F_k g_0|}{|F_k|} \| x \| \\
	& \stackrel{k \to \infty}{\to} 0 .
\end{align*}
We can thus infer that
$$\omega^{(K)}(x) - \omega^{(K)}(\Xi_{g_0} x) = \operatorname{sh} \left( \prescript{*}{}f_x(K) \right) - \operatorname{sh} \left( \prescript{*}{}f_{\Xi_{g_0} x}(K) \right) = \operatorname{sh} \left( \prescript{*}{}f_x (K) - \prescript{*}{}f_{ \Xi_{g_0} x }(K) \right) = 0 .$$
So $\omega^{(K)}$ is $\Xi$-invariant.

Our goal is to show that $\omega^{(K)}(\iota(a)) = m \left( a \vert \operatorname{Ann}(\ker \iota) \right)$. If we prove this, then we'll know that $\operatorname{sh} \left( \prescript{*}{}f_{\iota(a)}(K) \right) = m \left( a \vert \operatorname{Ann}(\ker \iota) \right)$ for all unlimited hypernaturals $K \in \prescript{*}{}{ \mathbb{N} }$, meaning that the sequence $\left( \sigma_k \left( \frac{1}{|F_k|} \sum_{g \in F_k } \Xi_g \iota(a) \right) \right)_{k \in \mathbb{N} } = \left( \left\| \frac{1}{|F_k|} \sum_{g \in F_k} \Xi_g \iota(a) \right\| \right)_{ k \in \mathbb{N} }$ converges to $m \left( a \vert \operatorname{Ann}(\ker \iota) \right)$.

Set $\phi = \omega^{(K)} \circ \iota \in \mathcal{S}^G$. We claim that $\phi$ is $\left( a \vert \operatorname{Ann}(\ker \iota) \right)$-maximizing. Since $\phi$ is a $\Theta$-invariant state on $\mathfrak{A}$, and vanishes on $\ker \iota$, we know that
$$\phi(a) \leq m \left( a \vert \operatorname{Ann}(\ker \iota) \right) .$$

Next, we prove the opposite inequality under the assumption that $(\mathfrak{A}, G, \Theta; \iota)$ is faithful. Since $\iota$ is injective, we know that it's an isometry. Let $\psi \in \mathcal{S}^G$. Then
\begin{align*}
	\psi(a)	& = \psi \left( \frac{1}{|F_k|} \sum_{g \in F_k} \Theta_g a \right) \\
	& \leq \left\| \frac{1}{|F_k|} \sum_{g \in F_k} \Theta_g a \right\| \\
	& = \left\| \frac{1}{|F_k|} \sum_{g \in F_k} \Xi_g \iota(a) \right\| .
\end{align*}
Therefore $\left\| \frac{1}{|F_k|} \sum_{g \in F_k} \Xi_g \iota(a) \right\| \in [\psi(a), \| a \|]$ for all $k \in \mathbb{N}$, so $\phi(a) \in [\psi(a) , \| a \|]$.
So
\begin{align*}
	\psi(a)	& \leq \phi(a) \\
	\Rightarrow \sup_{\psi \in \mathcal{S}^G} \psi(a)	& \leq \phi(a) .
\end{align*}
Thus
$$m \left( a \vert \mathcal{S}^G \right) = \sup_{\psi \in \mathcal{S}^G} \psi(a) \leq \phi(a) .$$
This establishes that $\phi$ is $\left( a \vert \operatorname{Ann}(\ker \iota) \right)$-maximizing when the C*-model is faithful, and in particular that $\phi(a) = \omega^{(K)}(\iota(a)) = m \left( a \vert \operatorname{Ann}(\ker \iota) \right)$.

Now, suppose that $(\mathfrak{A}, \mathbb{Z}, \Theta; \iota)$ is not necessarily faithful, and let $\left( \tilde{\mathfrak{A}} , G, \tilde{\Theta}; \tilde{\iota} \right)$ be the faithful C*-model given by
\begin{align*}
	\tilde{\mathfrak{A}}	& = \mathfrak{A} / \ker \iota , \\
	\tilde{\Theta}_g(a + \ker \iota)	& = \Theta_g a + \ker \iota , \\
	\tilde{\iota}(a + \ker \iota)	& = \iota(a) .
\end{align*}

Let $\tilde{\mathcal{S}}^G$ denote the family of $\tilde{\Theta}$-invariant states on $\tilde{\mathfrak{A}}$, and let $\pi : \mathfrak{A} \twoheadrightarrow \tilde{\mathfrak{A}}$ be the quotient map $\pi : a \mapsto a + \ker \iota$, so the following diagram commutes:
$$
\begin{tikzcd}
	\mathfrak{A} \arrow[r, two heads, "\pi"] \arrow[rd, "\iota"]	& \tilde{\mathfrak{A}} \arrow[d, hook, dotted, "\tilde{\iota}"] \\
	& \mathfrak{M}
\end{tikzcd}
$$
We claim that there's a bijective correspondence between $\tilde{\mathcal{S}}^G$ and $\operatorname{Ann}(\ker \iota)$. If $\psi$ is a $\tilde{\Theta}$-invariant state on $\tilde{\mathfrak{A}}$, then we can pull it back to a state $\psi_0$ on $\mathfrak{A}$ by
$$\psi_0 = \psi \circ \pi .$$
Then $\psi_0$ vanishes on $\ker \pi = \ker \iota$, so it's an element of $\operatorname{Ann}(\ker \iota)$. On the other hand, if we started with a $\Theta$-invariant state $\psi$ on $\mathfrak{A}$ that vanished on $\ker \iota$, then we could push it to a state $\tilde{\psi}$ on $\tilde{\mathfrak{A}}$ by defining
$$\tilde{\psi}(\pi(a)) = \psi(a) ,$$
which is well-defined because $\ker \iota \subseteq \ker \psi$. These two processes satisfy the identities
\begin{align*}
	\psi_0(a)	& = \psi(\pi(a)) , \\
	\psi(a)	& = \tilde{\psi}(\pi(a)) .
\end{align*}

Since $\left( \tilde{\mathfrak{A}} , G, \tilde{\Theta}; \tilde{\iota} \right)$ is faithful, we know from our previous argument that $\tilde{\phi} \in \tilde{\mathcal{S}}^G$ is $\left( \pi(a) \vert \tilde{\mathcal{S}}^G \right)$-maximizing, and therefore (due to the bijective correspondence between $\tilde{\mathcal{S}}^G$ and $\operatorname{Ann}(\ker \iota)$) we can infer that $\phi$ is $\left( a \vert \operatorname{Ann}(\ker \iota) \right)$-maximizing. Therefore
$$m \left( a \vert \operatorname{Ann}(\ker \iota) \right) = m \left( \pi(a) \vert \tilde{\mathcal{S}}^G \right) = \tilde{\phi} (\pi(a))) = \phi(a) = \omega^{(K)}(\iota(a)) .$$

Since we've shown that $\operatorname{sh} \left( \prescript{*}{}f_{\iota(a)}(K) \right) = m \left( a \vert \operatorname{Ann}(\ker \iota) \right)$ for all unlimited hypernaturals $K \in \prescript{*}{}{ \mathbb{N} } \setminus \mathbb{N}$, we can conclude that $\lim_{k \to \infty} \left\| \frac{1}{|F_k|} \sum_{g \in F_k} \Xi_g \iota(a) \right\|$ exists, and is equal to $m \left( a \vert \operatorname{Ann}(\ker \iota) \right)$.
\end{proof}

\section*{Acknowledgments}

This paper is written as part of the author's graduate studies. He is grateful to his beneficent advisor, professor Idris Assani, for no shortage of helpful guidance.

\bibliography{Bibliography}
\end{document}